\RequirePackage{fix-cm}
\documentclass[smallextended,numbook]{svjour3}       
\smartqed  
\usepackage{graphicx}
 \usepackage{mathptmx}      
\usepackage{latexsym}

\usepackage{microtype}
\usepackage{amsmath,amsfonts,amssymb}
\usepackage{pifont}

\newcommand{\eps}{\ensuremath{\varepsilon}}

\newcommand{\bean}{\begin{eqnarray*}}
\newcommand{\eean}{\end{eqnarray*}}

\newcommand{\WW}{\ensuremath{\mathbb{Z}_{0}}}

\newcommand{\mrw}{\ensuremath{\mathrm{w}}}
\newcommand{\mrrw}{\ensuremath{\mathrm{rw}}}

\newcommand{\NN}{\mathbb N}

\newcommand{\barr}{\begin{array}}
\newcommand{\earr}{\end{array}}

\newcommand{\be}{\begin{equation}}
\newcommand{\ee}{\end{equation}}
\newcommand{\bear}{\begin{eqnarray}}
\newcommand{\eear}{\end{eqnarray}}

\newcommand{\bfi}{\begin{figure}}
\newcommand{\efi}{\end{figure}}
\newcommand{\bc}{\begin{center}}
\newcommand{\ec}{\end{center}}

\newcommand{\beq}{\begin{eqnarray*}}
\newcommand{\eeq}{\end{eqnarray*}}

\journalname{Mathematics of Control, Signals and Systems 25(4):473--490, 2013 }

\begin{document}

\title{Two Extensions of Topological Feedback Entropy\thanks{This work was supported
by Australian Research Council grant DP110102401.}
}

\titlerunning{Topological Feedback Entropy Extensions}        

\author{Rika Hagihara         \and
        Girish N. Nair 
}

\authorrunning{R. Hagihara, G.~N. Nair} 

\institute{
           G.~N. Nair \at
Department of Electrical and Electronic Engineering, University of Melbourne, VIC 3010 Australia\\
Tel.: +61-3-8344-6701\\
Fax.:+61-3-8344-6678\\
\email{gnair@unimelb.edu.au}
}

\date{}

\maketitle

\begin{abstract}
Topological feedback entropy (TFE) measures the intrinsic rate at which a continuous, fully observed, deterministic
control system generates information for controlled set-invariance.
In this paper we generalise this notion in two directions;
one is to continuous, partially observed systems, and the other is to discontinuous, fully observed systems.
In each case we show that the corresponding generalised TFE coincides with
the smallest feedback bit rate that allows a form of controlled invariance to be achieved.

\end{abstract}


%
%




\section{Introduction}
\label{intro}

In 1965, Adler, Konheim, and McAndrew \cite{adler} introduced {\em topological entropy} as a measure of the fastest rate at which a continuous, discrete-time, dynamical system in a compact space generates initial-state information.
Though related to the measure-theoretic notion of {\em Kolmogorov-Sinai entropy} (see e.g. \cite{walters}),
it is a purely deterministic notion and requires only a topology on the state space, not an invariant measure.
Subsequently, Bowen \cite{bowenTAMS71} and Dinaburg \cite{dinaburgDokl70} proposed an alternative, metric based definition of topological entropy.
This accommodates uniformly continuous dynamics on noncompact spaces and is equivalent to the original definition on compact spaces.

These concepts play an important role in dynamical systems but remained largely neglected in control theory.
However, the emergence of digitally networked control systems (see e.g. \cite{antsaklisSpecial07}) over the last four decades
renewed interest in the information theory of feedback, and in 2004 the techniques of Adler et al. were adapted
to introduce the notion of {\em topological feedback entropy (TFE)} \cite{nairTAC03}.
Unlike topological entropy, TFE quantifies the {\em slowest} rate at which a continuous,
deterministic, discrete-time dynamical system {\em with inputs} (i.e. a control plant)
generates information, with states confined in a specified compact set.
Equivalently, it describes the rate at which the plant generates
information relevant to the control objective of set-invariance.

From an engineering perspective, the operational significance of TFE arises from the fact that it
coincides with the smallest average bit-rate between the plant and controller that allows set invariance to be achieved.
In other words, if an errorless digital channel with limited bit rate $R$ connects
the plant sensor to the controller, then a coder, controller and decoder that achieve set invariance can be constructed if and almost only if $R$
exceeds the TFE of the plant.\footnote{Without loss of generality, there can also be an errorless digital channel
from the controller to the plant actuator; in this case $R$ is taken to be the minimum of the two channel rates.}
Thus set invariance is possible if and almost only if the digital channel can transport information
faster than the plant generates it.


Later, the notion of {\em invariance entropy} was introduced for continuous, deterministic control systems in continuous time \cite{coloniusSICON09,kawanThesis,kawanDCDS11} based on the metric-space techniques of Bowen.
This measures the smallest growth rate of the number of open-loop control functions
needed to confine the states within an arbitrarily small distance from a given compact set.
In contrast, TFE is defined in a topological space and counts the minimum rate at which initial state uncertainty sets are refined,
with states confined to a given compact set.
Despite these significant conceptual differences, it has been established that TFE and invariance entropy are essentially the same object.

A limitation of the formulations above of entropy for control systems is their restriction to plants with fully observed states
and continuous dynamics.
This makes them inapplicable when only a function of the state can be measured or
the plant has discontinuities such as a quantised internal feedback loop.
A recent article \cite{coloniusMCSS11} studied continuous, partially observed plants
in continuous time, with the objective being to keep the plant outputs arbitrarily close
to a given compact set for any initial state in another set.
The notion of {\em output invariance entropy} was defined as a lower bound on the required data rate.
However, the control functions in this formulation are chosen according to the initial state.
Thus when the controller has access to only the output not the state,
the lower bound may be loose.

For discontinuous systems, notions of topological entropy have been proposed for piecewise continuous, piecewise monotone (PCPM) maps
on an interval \cite{misiurewicz92,kopf05}, using a Bowen-style metric approach.
It is shown that the topological entropy of a PCPM map coincides with the exponential growth rate of the number
of subintervals on which the iteration of the map is continuous and monotone.
In \cite{savkinAUTO06}, a metric approach was also adopted to define
a topological entropy for a possibly discontinuous, open-loop, discrete-time system driven by a sequence of disturbances.
However, it is not clear how these constructions can be adapted to feedback control systems with vector-valued states.

In this paper we use open-set techniques to extend TFE in two directions;
one is to continuous plants with continuous, partial observations in section \ref{output_section}, and the other is to a class of discontinuous plants with full state observations in section \ref{piecewisesec}.
In each case we show that the extended TFE coincides with the smallest average bit-rate
between the plant and controller that allows weak invariance to be achieved,
thus giving these concepts operational relevance in communication-limited control.

Though both generalisations involve open covers, the assumptions and
techniques underlying them are significantly different.
The questions of how to compute bounds on these notions and how to construct a unified notion of feedback entropy for
discontinuous, partially observed systems are ongoing areas of research.

\noindent {\em Terminology.} The nonnegative integers are denoted by $\WW$ and the positive integers by $\NN$.
Sequence segments $(x_s,\ldots, x_t)$ are denoted by $x_s^t$.
A collection $\alpha$ of open sets in a topological space $Z$ is called an {\em open cover}
of a set $W\subseteq Z$ if $\cup_{A\in\alpha}A\supseteq W$.
A subcollection $\beta\subseteq\alpha$ is called a {\em subcover} of $W$ if
$\cup_{B\in\beta}B\supseteq W$.
If $\alpha$ contains at least one finite subcover of $W$, then
$N(\alpha|W)\in\NN$ denotes the minimal cardinality over them all;
in the case where $W=Z$, the second argument will be dropped.

\section{Continuous, Partially Observed Systems}
\label{output_section}

In this section, we extend topological feedback entropy (TFE) to continuous, partially observed, discrete-time control systems.
We then prove that the TFE coincides with the smallest average feedback bit-rate
that allows a form of controlled set-invariance to be achieved via a digital channel.

\subsection{Weak Topological Feedback Entropy}
\label{output_WITFE}

To improve readability, most of the proofs in this subsection are deferred to the Appendix.
Let $X$ be a compact topological space and consider the continuous, partially observed, discrete-time control system
\begin{equation}
\begin{split}
x_{k+1}&=f(x_k, u_k)\in X \\
y_k&=g(x_k)\in Y
\end{split}
,\quad \forall k\in\WW,
\label{postate}
\end{equation}
where the input $u_k$ is taken from a set $U$, the output $y_k$ lies in a topological space $Y$,
and the functions $g$ and $f(\cdot, u)$, $u\in U$, are continuous.
For simplicity write
$f_u(\cdot):= f(\cdot , u)$,  $f_{u_0^{s-1}}:=f_{u_{s-1}}f_{u_{s-2}} \cdots f_{u_0}$ and for any $s\in\NN$, let $g_{u_0^{s-1}}$ denote
the continuous function that maps $x_0$ to $y_0^s$ when the plant is fed with input sequence $u_0^{s-1}$.
Given a compact target set $K \subsetneqq X$ with  nonempty interior,
assume the following:
\begin{itemize}
\item[(Ob)] The plant is \emph{uniformly controlled observable}: there exists $s \in \WW$ and an input sequence $v_0^{s-1}$
such that the map $g_{v_0^{s-1}}$  is continuously invertible on  $g_{v_0^{s-1}}(X)$.
\item[(WCI)] The set $K$ is \emph{weakly controlled invariant}: there exists $t \in \NN$
such that for any $x_0 \in X$ there exists a sequence $\{ H_k(x_0) \}_{k=0}^{t-1}$ of inputs in $U$ that ensures $x_t \in \mbox{int}K$.
\end{itemize}

Condition (Ob) states that there is a fixed input sequence $v_0^{s-1}$ that allows the initial state $x_0$ to be
recovered as a continuous function of the output sequence $y_0^{s}$. The states $x_1,\ldots, x_s$
are not required to lie inside $\mathrm{int}K$;
this gives the  freedom to trade transient control performance
off for improved accuracy in state estimation.

We also remark that the main  difference between WCI
and the usual definition of controlled set-invariance
is that the state only needs to be steerable to the target set in $t$ time steps, not one.
As a technicality, the topological methods we employ
also require the state $x_t$ to lie in the interior of $K$.

We now introduce tools to describe the information generation rate.
Pick $s \in \WW$ and an input sequence $v_0^{s-1}$.
Let $\alpha$ be an open cover of $g_{v_0^{s-1}}(X)\subseteq Y^{s+1}$, $\tau$ a positive integer,
and $G=\{ G_k: \alpha \to U \}_{k=0}^{\tau-1}$ a sequence of $\tau$ maps that assign input values to each element of $\alpha$.
Define $\mathcal{C}$ to be the set of all tuples $\left (s, v_0^{s-1}, \alpha, \tau, G\right )$
that satisfy the following constraint:
\begin{itemize}
\item[(C)] For any $A \in \alpha$ and $x_0 \in g_{v_0^{s-1}}^{-1}(A)$,
the concatenated input sequence $\left (v_0^{s-1},G(A)\right )$ yields $x_{s+\tau} \in \mbox{int}K$,
with $g_{v_0^{s-1}}$ continuously invertible on $g_{v_0^{s-1}}(X)$.
\end{itemize}



\begin{proposition}\label{output_feasibility}
If (Ob) and (WCI) hold, then  $\mathcal{C}\neq\emptyset$, i.e. constraint (C) is feasible.
\end{proposition}

Next, we use $\alpha$ to track the orbits of the initial states.
Divide time up into cycles of duration $s+\tau$
and apply an input sequence $v_0^{s-1}$ that satisfies (Ob) for the first $s$ instants
of each cycle.
Let $A_0, A_1, \dots$ be elements of $\alpha$
and for each $j \in \NN$ define
\begin{equation}\label{B_j}
B_j :=\left\{ x_0 \in X:
\begin{array}{l} x_{i(s+\tau)} \in g_{v_0^{s-1}}^{-1}(A_i), \, 0 \leq i \leq j-1, \, \mbox{and} \\
                        u_{i(s+\tau)}^{(i+1)(s+\tau)-1}= (v_0^{s-1}, G(A_i)), \, 0 \leq i \leq j-2
\end{array} \right\}. 
\end{equation}
In other words,  $B_j$ is the set of initial states
such that during the $(i+1)$th cycle,
the sequence $y_{i(s+\tau)}^{i(s+\tau)+s}$ of the first $s+1$ outputs
lies inside $A_i$ when the sequence of $s+\tau$ inputs over the cycle is
$\left (v_0^{s-1}, G(A_i)\right )\in U^{s+\tau}$ for every $i\in[0,\ldots, j-2]$.
\begin{proposition}\label{B_j_properties}
The set $B_j$ has the following properties.
\begin{enumerate}
 \item Each $B_j$ is an open set.
 \item Every $x_0 \in X$ must lie in some $B_j$.
\end{enumerate}
\end{proposition}

\begin{corollary}
For each $j \in \NN$ the collection of sets $B_j$,
\begin{equation}\label{beta_j}
\beta_j := \{ B_j: A_0, A_1, \dots , A_{j-1} \in \alpha \},
\end{equation}
is an open cover of $X$.
\end{corollary}

Now, since $\beta_j$ is an open cover of $X$, compactness implies that it must contain a finite subcover.
Consider a minimal subcover with cardinality $N(\beta_j)\in\NN$.
As no set in this minimal subcover of $\beta_j$ is contained in a union of other sets, each carries new information.
Thus as $N(\beta_j)$ increases, the amount of information gained about the initial state grows.
In order to measure the asymptotic rate of information generation, we need the following:
\begin{lemma}\label{subadditivity}
The sequence ${(\log_2 N(\beta_j))}_{j=1}^{\infty}$ is subadditive.
\end{lemma}
The next proposition follows from Fekete's lemma; see e.g. \cite[Theorem~4.9]{walters} for a proof.
\begin{proposition}\label{inf}
The following limit exists and equals the right hand infimum:
\begin{equation}\label{limit1}
\lim_{j \to \infty}\frac{\log_2 N(\beta_j)}{j(s+\tau)}
=\inf_{j \in \mathbb{N}} \frac{\log_2 N(\beta_j)}{j(s+\tau)}.
\end{equation}
\end{proposition}

We now define a topological feedback entropy for continuous, partially observed plants:
\begin{definition}
The \emph{weak topological feedback entropy (WTFE)} of the plant (\ref{postate})
with target set $K$ is defined as
\begin{equation}
h_{\mbox{w}}
:=\inf_{\left (s, v_0^{s-1},\alpha, \tau, G\right )\in\mathcal{C}}\lim_{j \to \infty} \frac{\log_2 N(\beta_j)}{j(s+\tau)}
\stackrel{(\ref{limit1})}{=}
\inf_{\left (s, v_0^{s-1},\alpha, \tau, G\right )\in\mathcal{C}, j\in\NN}
\frac{\log_2N(\beta_j)}{j(s+\tau) },
\label{WITFE}
\end{equation}
where $\mathcal{C}$ is the set of all tuples
$\left (s, v_0^{s-1},\alpha, \tau, G\right )$ that satisfy constraint (C).
\end{definition}


This definition reduces to the {\em weak invariance TFE} of \cite{nairTAC03} if the plant is fully observed, since in that case condition (Ob) is trivially satisfied with $s=0$ and $g_{v_0^{s-1}}=g$ reducing to the identity map.
Like classical topological entropy for dynamical systems \cite{adler},
it measures the rate at which initial state uncertainty sets are refined as more and more observations are taken `via' an open cover.
The differences here are that control inputs must be accommodated, only partial state observations  are possible, and the slowest rate is of interest, not the fastest.

Instead of `pulling back' and intersecting the open sets $A_0,A_1,\ldots\in\alpha$ to form  an open cover $\beta_j$ of the initial state space $X$,
suppose we simply counted the smallest cardinality of minimal subcovers of
the open cover $\alpha$ itself, under constraint (C).
It turns out that this more direct approach yields the same number:

\begin{proposition}\label{inf_equation}
The weak topological feedback entropy (\ref{WITFE}) for the continuous, partially observed
plant (\ref{postate}) satisfies the identities
\begin{equation*}\label{inf_equation_formulas}
\begin{split}
h_{\mbox{w}}
&=\inf_{\left (s,v_0^{s-1},\alpha, \tau, G\right )\in\mathcal{C}} \frac{\log_2N(\beta_1)}{s+\tau} \\
&=\inf_{\left (s,v_0^{s-1},\alpha, \tau, G\right )\in\mathcal{C}} \frac{\log_2 N(\alpha| g_{v_0^{s-1}}(X))}{s+\tau},
\end{split}
\end{equation*}
\end{proposition}

\begin{proof}
Note first that by plugging $j=1$ into the second equality in (\ref{WITFE}) we obtain
\begin{equation}
h_{\mbox{w}}\label{rewrite1}
= \inf_{(s,v_0^{s-1}, \alpha, \tau, G)\in\mathcal{C}, j\in\NN }\frac{\log_2 N(\beta_j)}{j(s+\tau)}
\leq  \inf_{(s,v_0^{s-1}, \alpha, \tau, G)\in\mathcal{C}}\frac{\log_2 N(\beta_1)}{s+\tau}.
\end{equation}

We prove that the infimum (with $j=1$) arbitrarily close to $h_{\mbox{w}}$ can be achieved.
The definition of the WTFE, combined with the fact that $\lim_{j \to \infty}j/(j-1)=1$ and Proposition \ref{inf},
implies that given $\eps >0$ we can find a tuple $(s, v_0^{s-1}, \alpha, \tau, G)$ that satisfies constraint (C) and a large $j \in \NN$ such that
\begin{equation}\label{rewrite2}
h_{\mbox{w}}
\leq \frac{\log_2 N(\beta_j)}{(j-1)(s+\tau)}
= \frac{\log_2 N(\beta_j)}{j(s+\tau)} \frac{j}{j-1}
\leq h_{\mbox{w}} +\eps.
\end{equation}
We now construct a new tuple $(s', {v'}_0^{s-1}, \alpha', \tau', G')$.
Let $s'$ and ${v'}_0^{s-1}$ be $s$ and $v_0^{s-1}$ chosen above.
The open cover $\alpha'$ of $g_{v_0^{s-1}}(X)$ in $Y^{s+1}$ is obtained as follows.
Recall that $\beta_j$ is an open cover of $X$.
We use the continuous map $g_{v_0^{s-1}}^{-1}$ to form a collection of open sets in $g_{v_0^{s-1}}(X)$, $\{ g_{v_0^{s-1}}(B_j): B_j \in \beta_j \}$.
Note that the union of sets $g_{v_0^{s-1}}(B_j)$ is $g_{v_0^{s-1}}(X)$.
Since $g_{v_0^{s-1}}(X)$ is equipped with the subspace topology of $Y^{s+1}$, each open set $g_{v_0^{s-1}}(B_j)$ in $g_{v_0^{s-1}}(X)$ can be expanded to an open set $B_j'$ in $Y^{s+1}$ so that $g_{v_0^{s-1}}^{-1}$ maps points in $B_j' \cap g_{v_0^{s-1}}(X)$ into $B_j$.
Now $\alpha'={\{ B_j': B_j \in \beta_j \}}$ is an open cover of $g_{v_0^{s-1}}(X)$ in $Y^{s+1}$.
Recall that each $B_j \in \beta_j$ is of the form
\begin{multline*}
B_j=g_{v_0^{s-1}}^{-1}(A_0) \cap \Phi_{G(A_0)}^{-1} g_{v_0^{s-1}}^{-1}(A_1) \cap \Phi_{G(A_0)}^{-1} \Phi_{G(A_1)}^{-1} g_{v_0^{s-1}}^{-1}(A_2) \cap \cdots \\
\dots \cap \Phi_{G(A_0)}^{-1} \dots \Phi_{G(A_{j-2})}^{-1} g_{v_0^{s-1}}^{-1}(A_{j-1}),
\end{multline*}
where $\Phi_{G(A)}= f_{v_0^{s-1}G(A)}$ with $A_0, \dots, A_{j-1}$ in $\alpha$.
Set $\tau'=(j-1)(s+\tau)-s$ and define a sequence of $\tau'$ maps $G'$ by
$G'(B_j')=G(A_0)v_0^{s-1}G(A_1)\dots v_0^{s-1}G(A_{j-2})$.
Since $(s, v_0^{s-1}, \alpha, \tau, G)$ satisfies constraint (C), by construction $(s', {v'}_0^{s-1}, \alpha', \tau', G')$ satisfies (C).
Clearly, $N(\beta_1') = N(\beta_j)$, where $\beta_1' = g_{v_0^{s-1}}^{-1}(\alpha')$, and from inequalities in (\ref{rewrite2}) we have
\[ h_{\mbox{w}} \leq \frac{\log_2 N(\beta_1')}{s'+\tau'} \leq h_{\mbox{w}}+\eps. \]
Since $\eps >0$ was arbitrary, the above inequalities and the earlier result in (\ref{rewrite1}) give the first equality.

For the second equality, observe that given any tuple $(s, v_0^{s-1}, \alpha, \tau, G)$ satisfying constraint (C)
we have $N(\beta_1)\equiv N(\beta_1|X)=N(\alpha | g_{v_0^{s-1}}(X))$,
where $\beta_1=g_{v_0^{s-1}}^{-1}(\alpha)$, since $\alpha$ is an open cover of
$g_{v_0^{s-1}}(X)$ in $Y^{s+1}$, $g_{v_0^{s-1}}(X)$ is equipped with the subspace topology of $Y^{s+1}$,
and $g_{v_0^{s-1}}$ is a homeomorphism from the compact space $X$ onto $g_{v_0^{s-1}}(X)$. Thus we obtain the desired equation.
\qed
\end{proof}

The first equality in this result is a technical simplification that allows the cycle index $j\in\NN $ in the definition of WTFE
to be restricted to the value 1. The second equality follows almost immediately but is
conceptually  more significant. In rough terms, it states that under constraint (C), the smallest
growth rate of the number of `topologically distinguishable' output sequences in $g_{v_0^{s-1}}(X)\subseteq Y^{s+1}$
coincides with the smallest growth rate of the number of such initial states in $X$.

\subsection{Data-Rate-Limited Weak Invariance}

The weak topological feedback entropy (WTFE) constructed in the previous subsection
is defined in abstract terms. We now show its relevance to the problem of
feedback control via a channel with finite bit-rate, for a general class of coding and control laws.

Suppose that the sensor that measures the outputs of  the plant (\ref{postate}) transmits one discrete-valued symbol $s_k$ per sampling interval to the controller over a digital channel.
Each symbol transmitted by the coder may potentially depend on all past and present outputs and past symbols,
\begin{equation}\label{symbol}
s_k=\gamma_k({\{ y_i\}}_{i=0}^{k}, {\{s_i \}}_{i=0}^{k-1})\in S_k, \quad \forall k \in \WW,
\end{equation}
where $S_k$  is a coding alphabet $S_k$ of time-varying size $\mu_k$ and
$\gamma_k: Y^{k+1}\times S_0 \times \dots \times S_{k-1} \to S_k$ is the coder map at time $k$.
Assuming that the digital channel is errorless, at time $k$ the controller has $s_0, \dots, s_k$ available and generates
\begin{equation}\label{control_input}
u_k=\delta_k({\{s_i\}}_{i=0}^{k}) \in U, \quad \forall k \in \WW,
\end{equation}
where $\delta_k: S_0 \times \dots \times S_k \to U$ is the controller map at time $k$.
Define the \emph{coder-controller} as the triple
$(S, \gamma, \delta):=( {\{ S_k \}}_{k \in \WW}$, ${\{ \gamma_k \}}_{k \in \WW}$, ${\{ \delta_k \}}_{k \in \WW} )$
of the alphabet, coder and controller sequences.

Suppose that the performance objective of the coder-controller is to render $K$ \emph{weakly invariant},
i.e. to ensure that there exists a time $q \in \NN$ such that for any $x_0 \in X$, $x_{q} \in \mathrm{int}K$.
Let $Q_\mrw\subseteq\NN$ be the set of all the invariance times $q$ of a given coder-controller
that achieves this objective.
For each $q\in Q_\mrw$, a $q$-periodic coder-controller extension that ensures $x_{q}, x_{2q}, \ldots \in \mathrm{int}K$ can be constructed by `resetting the clock' to zero at times $q,2q,\ldots$.\footnote{See section III in \cite{nairTAC03}.}
The average transmission data rate of each $q$-periodic extension is simply $\frac{1}{q} \sum_{j=0}^{q-1} \log_2 \mu_j$,
and the (smallest) average data rate required by the coder-controller is then
\begin{equation}\label{data_rate}
R:=\inf_{q\in Q_\mrw} \frac{1}{q} \sum_{j=0}^{q-1} \log_2 \mu_j \quad \mbox{(bits/sample)}.
\end{equation}

We remark that in \cite{nairTAC03}, the communication requirements of the coder-controller
were measured instead by the asymptotic average data rate, i.e.  over all $j\in\WW$, and the
weak-invariance control objective was to ensure $x_q,x_{2q},\ldots\in\mathrm{int}K$.
Thus for that coder-controller, $q\NN\subseteq Q_\mrw$, and  applying  (\ref{data_rate}) to it would
 yield a less conservative  number.

The main result of this subsection follows.
\begin{theorem}
Consider the continuous, partially observed, discrete-time plant (\ref{postate}).
Suppose that the uniformly controlled observability (Ob) and weak controlled invariance (WCI) conditions hold
for a given target set $K$.

For $K$ to be made weakly invariant by a coder-controller of the form (\ref{symbol}) and (\ref{control_input}),
the feedback data rate in $R$ (\ref{data_rate}) cannot be less than the weak topological feedback entropy:
\begin{equation*}\label{pctheorem}
R \geq h_{\mrw}.
\end{equation*}
Furthermore, this lower bound is tight: there exist coder-controllers that achieve weak invariance at data rates arbitrarily close to $h_{\mrw}$.
\end{theorem}

This theorem says that weak invariance can be achieved by some coder-controller if and (almost) only if the available data rate
exceeds the WTFE of the plant on the target set.
This gives  operational significance to WTFE and justifies its intepretation as the rate at which
the plant generates information for weak-invariance.

The remainder of this subsection comprises the proof of this result.

\paragraph{Necessity of the Lower Bound.}

Let the coder-controller $(S, \gamma, \delta)$ achieve invariance with data rate $R$. By (\ref{data_rate}),
$\forall\eps>0$ there exists $q \in Q_\mrw$ such that $x_{q} \in \mbox{int}K$ for any $x_0\in X$ and
\begin{equation}\label{approximation}
\frac{1}{q}\sum_{k=0}^{q-1}\log_2 \mu_k
<R+\eps.
\end{equation}
Let $(S^1,\gamma^1,\delta^1)$ be the $q$-periodic extension,
which achieves weak invariance at times $q,2q,\ldots$.

We wish to transform this periodic coder-controller into another that more closely matches
the structure of a WTFE quintuple and has nearly the same data rate.
Specifically, we seek a periodic coder-controller each cycle of which has two phases.
The first phase is  an initial `observation' phase, during which the controller applies a pre-agreed input sequence so as to enable the coder
to determine the exact value of the current state.
This is followed by an `action' phase, where the current state value is used to generate symbols and controls so as to achieve weak invariance by the end of the cycle.

Let $s \in \NN$ and the input sequence $v_0^{s-1}$ satisfy condition (Ob).
Set $\tau=q$ and construct a $(s+\tau)$-periodic  coder-controller $(S^2, \gamma^2, \delta^2)$
as follows; to simplify notation, the coding and control laws are defined only for the first cycle.
First, the controller applies the input sequence $u_0^{s-1}=v_0^{s-1}$;
during this time, the coder transmits an `empty' symbol.
For the remainder of the cycle, i.e. $s\leq k\leq s+\tau-1$, the coder-controller implements  $(S^1, \gamma^1, \delta^1)$
and obtains a new state $x_{s+\tau}\in \mbox{int}K$.
This can be achieved by defining the coder-controller $(S^2, \gamma^2, \delta^2)$ by
\begin{equation*}
\begin{split}
S^2_{k}&=
 \begin{cases}
 \{0\} & \mbox{if  $0\leq k\leq s-1$}\\
 S_{k-s}^1 & \mbox{if $s \leq k \leq s+\tau-1$}
 \end{cases}\\
s_{k} \equiv \gamma^2_{k}(y_0^k, s_0^{k-1}) &=
 \begin{cases}
 0 &  \mbox{if  $0\leq k\leq s-1$}\\
  \gamma^1_{k-s} (y_s^k, s_s^{k-1})
  &  \mbox{if $s \leq k \leq s+\tau-1$}
 \end{cases}, \\
u_{k}\equiv \delta^2_{k}(s_0^{k})&=
 \begin{cases}
 v_k & \mbox{if  $0\leq k\leq s-1$}\\
  \delta^1_{k-s} (s_s^{k})   & \mbox{if $s \leq k \leq s+\tau-1$}
 \end{cases}
\end{split}
\end{equation*}

Observe that the symbol sequence $s_s^{s+\tau-1}$ is completely determined by $x_{s}$ by a fixed map that incorporates both coder and controller laws.
Thus there exist maps $\Gamma$ and $\Delta$ such that
\begin{equation}\label{p_coder_controller}
\begin{split}
 s_s^{s+\tau-1}&\equiv \Gamma (x_{s}), \\
u_s^{s+\tau-1} &\equiv \Delta \big( s_s^{s+\tau-1}\big).
\end{split}
\end{equation}

Now consider the disjoint coding regions $\Gamma^{-1}(c_s^{s+\tau-1})$,
$c_s^{s+\tau-1}\in$ $S^2_s \times \cdots \times S^2_{s+\tau-1}$.
The total number of distinct symbol sequences is
$\prod_{k=s}^{s+\tau-1}\mu_{k-s}= \prod_{k=0}^{\tau-1}\mu_{k}$,
which must be greater than or equal to the number $n$ of distinct coding regions.
Denote these coding regions by $C^1, \dots, C^n$ and note that $f_{v_0^{s-1}}(X) \subseteq \cup_{i=1}^n C^i$.
We can now rewrite the control law in (\ref{p_coder_controller}) as
\begin{equation}\label{new_p_coer_controller}
u_s^{s+\tau-1}=\Delta^*\big( C^i \big) \quad \mbox{if} \quad x_{s} \in C^i,
\end{equation}
where $\Delta^*(C^i)=\Delta(c_s^{s+\tau-1})$ iff $C^i=\Gamma^{-1}(c_s^{s+\tau-1})$.

We now construct  the open cover $\alpha$ of $g_{v_0^{s-1}}(X)$
and the mapping sequence $G={\{ G_{k}: \alpha \to U \}}_{k=0}^{\tau-1}$ required in the definition of the WTFE.
To define $\alpha$,  first construct an open cover of $X$ as follows.
Observe that for any $x_s$ in region $C^i$,
\[ \Phi_{\Delta^*(C^i)}(x_s)\in \mbox{int}K,\]
where the left hand side denotes $f_{u_s^{s+\tau-1}}$ with the control law (\ref{new_p_coer_controller}).
By the continuity of $f_u$ (hence of $\Phi_{\Delta^*(C^i)}$) and the openness of $\mbox{int}K$,
 it then follows that for any $x_s \in C^i$ there is an open set $O(x_s)$ that contains $x_s$ and is such that
\[ \Phi_{\Delta^*(C^i)}(x)\in \mbox{int}K , \ \ \forall x \in O(x_s). \]
We can construct for each $i\in[1,\ldots , n]$ an open set $D^i=\cup_{x_s \in C^i} O(x_s)$ so that
\begin{equation*}\label{open_sets}
\Phi_{\Delta^*(C^i)}(x) \in \mbox{int}K, \ \ \forall  x \in D^i,
\end{equation*}
which is equivalent to
\begin{equation*}
\Phi_{\Delta^*(C^i)}(D^i) \subseteq \mbox{int}K.
\end{equation*}
We then use the continuous map
$h_{v_0^{s-1}}=f_{v_0^{s-1}}g_{v_0^{s-1}}^{-1}$ to form the
collection $\{ h_{v_0^{s-1}}^{-1}(D^i) \}_{i=1}^n$ of open sets in $g_{v_0^{s-1}}(X)$.
As $g_{v_0^{s-1}}(X)$ is equipped with the subspace topology of $Y^{s+1}$, each open set
$h_{v_0^{s-1}}^{-1}(D^i)$ in $g_{v_0^{s-1}}(X)$ can be expanded to
an open set $L^i$ in $Y^{s+1}$ so that $h_{v_0^{s-1}}$ maps points in $L^i \cap g_{v_0^{s-1}}(X)$ into $D^i$.
The open cover $\alpha$ for $g_{v_0^{s-1}}(X)$ is defined by $\alpha = \{ L^1, \dots, L^n \}$.
Finally, construct the mapping sequence $G$ on $\alpha$ by
\[ G_{k}(L^i)=\,\mbox{the $k$th element of $\Delta^*(C^i)$ \quad for each $L^i \in \alpha$ and $0 \leq k \leq \tau -1$}.    \]
It is evident that constraint (C) on $s, v_0^{s-1}, \alpha, \tau, G$ is satisfied.

The minimum cardinality $N(\alpha|g_{v_0^{s-1}}(X))$ of subcovers of $\alpha$ does not exceed the number $n$ of sets $L^i$ in $\alpha$.
From this we obtain
\begin{equation*}
\begin{split}
h_{\mrw}
&=\inf_{s, v_0^{s-1}, \alpha, \tau, G} \frac{\log_2 N(\alpha| g_{v_0^{s-1}}(X))}{s+\tau} \\
&\leq \frac{\log_2 N(\alpha | g_{v_0^{s-1}}(X))}{s+\tau} \\
&\leq \frac{\log_2 n}{s+\tau} \\
&\leq \frac{\log_2 \left( \prod_{k =0}^{\tau-1} \mu_{k } \right)}{s+\tau}
=\frac{1}{s+\tau}\sum_{k =0}^{\tau-1}\log_2 \mu_{k }
\leq \frac{1}{\tau}\sum_{k =0}^{\tau-1}\log_2 \mu_{k }.
\end{split}
\end{equation*}
\noindent From (\ref{approximation})
the last term of the right hand side is less
than $R+ \eps$, since $\tau=q$.
Hence $R \geq h_{\mrw}$.

\paragraph{Achievability of the Lower Bound}

To prove that a data rate arbitrarily close to $h_{\mrw}$ can be achieved,
we show that any tuple $(s, v_0^{s-1}, \alpha, \tau, G)$ that satisfies constraint (C) yields
an $(s+\tau)$-periodic coder-controller that renders $K$ weakly invariant in every cycle.

Let
\begin{equation*}\label{uninfimised_entropy}
H:=\frac{\log_2 N(\alpha|g_{v_0^{s-1}}(X))}{s+\tau}.
\end{equation*}
Recalling that $\alpha$ is an open cover of the compact set $g_{v_0^{s-1}}(X)$,
select a minimal subcover of $\alpha$ and denote it by $\{ A^1, \dots, A^m \}$, where $m=N(\alpha|g_{v_0^{s-1}}(X))$.
Since $g_{v_0^{s-1}}$ is a homeomorphism from $X$ onto $g_{v_0^{s-1}}(X)$ and $g_{v_0^{s-1}}(X)$ is equipped with the subspace topology of $Y^{s+1}$,
we see that $\{ g_{v_0^{s-1}}^{-1}(A^i)\}_{i=1}^m$ is a finite open cover of $X$.
We construct an $(s+\tau)$-periodic coding law using these overlapping sets as follows; to simplify notation, the coding and control laws are described only for the first cycle.

For each $k\in[0, s+\tau-1]$, set $S_k = \{1,\ldots , m\}$ if $k=s$  and $S_k=\{0 \}$ otherwise.
Then let
\begin{equation}\label{coding_rule}
s_k=
\begin{cases}
\min \{ i: x_{0} \in g_{v_0^{s-1}}^{-1}(A^i) \} &\mbox{if $k=s$} \\
0 & \mbox{otherwise}
\end{cases}.
\end{equation}
The coding alphabet is of size $\mu_k=m$ when $k=s$ and $1$ otherwise.

The next step is to construct the controller from the input sequence $v_0^{s-1}$ and mapping sequence $G$.
For the first $s$ instants of the cycle, the controller applies inputs $u_0^{s-1}=v_0^{s-1}$.
At time $s$, the coder determines the initial state $x_0=g_{v_0^{s-1}}^{-1}(y_0^s)$.
Upon receiving the symbol $s_{s}$ that indexes an open set $g_{v_0^{s-1}}^{-1}(A^i)$ containing $x_{0}$,
the controller applies control inputs via the rule $u_s^{s+\tau-1}=G(A^i)$.

The coder-controller thus constructed has period $s+\tau$.
By constraint (C) we have $x_{s+\tau} \in \mbox{int}K$, hence weak $(s+\tau)$-invariance is achieved.

By (\ref{coding_rule}), the average data rate over the cycle is
\begin{equation*}
\bar{R} =\frac{\log_2 m}{s+\tau}=\frac{\log_2 N(\alpha|g_{v_0^{s-1}}(X))}{s+\tau} \equiv H.
\end{equation*}
As $h_{\mrw}$ is the infimum of $H$, for any $\eps >0$ we can find $(s, v_0^{s-1}, \alpha, \tau, G)$ yielding $H < h_{\mrw}+\eps$.
Hence $\bar{R}<h_{\mrw}+\eps$. The result follows by observing that $R\leq\bar{R}$ by
definition (\ref{data_rate})  and then choosing $\eps$ arbitrarily small.

\section{Piecewise Continuous, Fully Observed Plants}
\label{piecewisesec}

In this section we introduce the  notion of {\em robust weak topological feedback entropy} for
a class of piecewise continuous, fully observed plants.
We then establish the relevance of this notion for bit-rate limited control.

\subsection{Robust Weak Topological Feedback Entropy}

As before, let $X$ be a compact topological space.
We consider the fully observed, piecewise continuous plant
\begin{equation}\label{pcstate}
x_{k+1}=F(x_k, u_k)\in X,
\ \ \forall k\in\WW,
\end{equation}
where the input $u_k$ is taken from a set $U$.
For simplicity write $F_u:= F( \cdot, u)$ and $F_{u_0^k}:= F_{u_k}\cdots F_{u_0}$.
Assume that there exists a finite partition $\mathcal{P}$ of $X$
and that each connected component of any element of $\mathcal{P}$ has nonempty interior.
For each $u \in U$ the map $F_u$ is continuous on each element of $\mathcal{P}$, i.e. $F_u$ is piecewise continuous.

Let $K$ be a compact target set with nonempty interior,
and impose the following condition on the plant:
\begin{itemize}
\item[(RWI)] The compact set $K$ can be made \emph{robustly weakly controlled invariant} under $F$:
there exists $t \in \NN$
such that for any $x \in X$,
 there exists an open set $O(x)$ containing $x$ and an input sequence $\{ H_k(x) \}_{k=0}^{t-1}$ in $U$
that ensures $x_t \in \mbox{int}K$ for any $x_0 \in O(x)$.
\end{itemize}

Note that this is stricter than the weak controlled invariance (WCI) condition in the previous section.
It restricts our focus to plants with states that can be driven in finite time into the interior of
$K$,
despite an arbitrarily small error in the  initial state measurement.
This requirement is not satisfied by all discontinuous plants, but is reasonable
if for instance the plant consists of an underlying continuous
system  with an `inner' quantised control loop and an `outer' control loop to be designed,
i.e. $F(x,u)\equiv f(x,w,u)$, where the inner control input $w$ is a quantised function  of $x$.
It can be shown that if $f$ is continuous and $K$ is WCI,
then it is always possible to find a sequence of quantised {\em joint} inputs
$(w,u)\equiv(q_1(x),q_2(x))$ that preserves WCI in the presence of small state measurement errors.
Condition (RWI) then applies if the inner loop control law is chosen to be such a robust law $q_1(x)$.
For reasons of space, details are omitted.

We now introduce a feedback entropy concept for this system.
The key difference from Section \ref{output_section} is that the output map $g$ is the identity and so the index $s$ in condition (Ob) there
is 0. This leads to taking an open cover $\alpha$ directly on $X$.
We form a triple $(\alpha, \tau, G)$, where $\tau \in \NN$ and $G=\{ G_k: \alpha \to U \}_{k=0}^{\tau-1}$
is a sequence of $\tau$ maps that assign input variables to all elements of $\alpha$.
Let $\mathcal{RC}$ be the set of all triples  $(\alpha, \tau, G)$
that satisfy the following constraint:
\begin{itemize}
\item[(RC)]
For any $A \in \alpha$ and $x_0 \in A$, the input sequence $G(A)$ yields $x_{\tau} \in \mbox{int}K$, i.e.
\[ F_{G_{\tau-1}(A)} \dots F_{G_0(A)}(A) \subseteq \mbox{int}K. \]
\end{itemize}

\begin{proposition}
If condition (RWI) is assumed, then constraint (RC) is feasible, i.e. $\mathcal{RC}\neq\emptyset$.
\end{proposition}

\begin{proof}
We give a complete proof to emphasise the difference from the setting in Section \ref{output_section}, cf.~Proposition \ref{output_feasibility}.
Assume (RWI).
We construct $(\alpha, \tau, G)$ that satisfies (RC) as follows.
Set $\tau$ equal to $t$ in (RWI).
By (RWI), for each $x \in X$ there is an open set $O(x)$ that contains $x$ and is such that $F_{H_{\tau-1}(x)} \dots F_{H_0(x)}(O(x)) \subseteq \mbox{int}K$.
By ranging $x$ in $X$ we form an open cover of $X$, $\{ O(x): x \in X \}$.
The compactness of $X$ implies that there exists a finite subcover $\alpha=\{ O (x_q) \}_{q=1}^r$.
Construct the sequence $G=\{ G_k \}_{k=0}^{\tau-1}$ of $\tau$ maps on $\alpha$ by
\begin{equation*}
G_k(O(x_q))=H_k(x_q) \quad  \mbox{for each $1 \leq q \leq r$ and $0 \leq k \leq \tau-1$.}
\end{equation*}
By construction we have that if $x \in O(x_q)$ for some $1 \leq q \leq r$ then
\[ F_{G(O(x_q))}(x)=F_{\{ H_k(x_q) \}_{k=0}^{\tau-1}}(x) \in \mbox{int}K. \]
This confirms the feasibility of (RC) under (RWI).
\qed
\end{proof}

Motivated by Proposition \ref{inf_equation}, we introduce the following concept:
\begin{definition}
The \emph{robust weak topological feedback entropy (RWTFE)} of the plant (\ref{pcstate}) is
\begin{equation}\label{RWITFE}
h_{\mrrw}:=\inf_{(\alpha, \tau, G)\in\mathcal{RC}}\frac{\log_2N(\alpha)}{\tau}.
\end{equation}
\end{definition}

Unlike TFE and the classical topological entropy, this definition does not involve
pulling back and intersecting the sets in $\alpha$ to  measure the rate
at which initial state uncertainty sets are refined.
Thus  it cannot be viewed {\em a priori} as an index of the rate at which the plant generates initial-state information.
However, such an interpretation becomes plausible due to the results in the next subsection.

\subsection{Robust Weak Invariance Under a Data-Rate Constraint}

We now show the relevance of robust weak topological entropy to the problem of
feedback control via a channel with finite bit-rate.

Consider the fully observed, piecewise continuous plant (\ref{pcstate}).
We now introduce a feedback loop with a coder-controller of the general form (\ref{symbol})--(\ref{control_input}),
but with the outputs $y_i$  replaced by states $x_i$.
The performance objective of the coder-controller is to render  $K$ \emph{robustly weakly invariant}, that is, to guarantee that $\exists q \in \NN$ such that $\forall x_0 \in X$, there is an open neighbourhood $O(x_0)\ni x_0$ with the property that the input sequence $u_0^{q-1}$ generated by the coder-controller $(S, \gamma, \delta)$ {\em acting on} $x_0$ ensures $F_{u_0^{q-1}} (x) \in \mbox{int}K$,  $\forall x \in O(x_0)$.
In other words, we desire that the control inputs generated by  (\ref{symbol})--(\ref{control_input}) for a plant with nominal initial state $x_0$ should still succeed in achieving weak $q$-invariance if the true initial state $x$ is sufficiently  close.

Let $Q_{\mrrw}\subseteq\NN$ be the set of all the robust weak invariance times $q$ of a given coder-controller that achieves this objective.
For  each $q\in Q_{\mrrw}$, we can construct a $q$-periodic coder-controller extension by `resetting the clock' to zero
at times $q,2q,\ldots$.\footnote{See section III in \cite{nairTAC03}.}
As before, the (smallest) average data rate required by the coder-controller is
then
\begin{equation}\label{r_data_rate}
R:=\inf_{q\in Q_{\mrrw}} \frac{1}{q} \sum_{j=0}^{q-1} \log_2 \mu_j.
\end{equation}


The main result of this subsection follows:
\begin{theorem}
Consider the piecewise continuous, fully observed plant (\ref{pcstate}) with a compact target set $K$ having nonempty interior.
Assume that initial states are in the compact topological space $X$.
Suppose that the robust weak invariance condition (RWI) holds on $F, K, U$.

For $K$ to be made robustly weakly invariant
by a coder-controller of the form (\ref{symbol})--(\ref{control_input}), the feedback data rate $R$ in (\ref{r_data_rate})
cannot be less than the robust weak topological feedback entropy (\ref{RWITFE}):
\begin{equation*}\label{rtheorem}
R\geq h_{\mrrw}.
\end{equation*}
Furthermore, this lower bound is tight: there exist coder-controllers that achieve robust weak invariance
at data rates arbitrarily close to $h_{\mrrw}$.
\end{theorem}

The rest of this subsection consists of the proof of this result.

\paragraph{Necessity of the Lower Bound.}

Pick an arbitrarily small $\eps>0$.
Given  a coder-controller $(S, \gamma, \delta)$ that achieves robust weak invariance,
there exists a robust weak invariance time $q \in \NN$ such that
\[ \frac{1}{q} \sum_{k=0}^{q-1}\log_2 \mu_k \leq R +\eps. \]
Set $\tau=q$ and  note that the symbol  sequence $s_0^{\tau-1}$
is completely determined by the initial state $x_0$,
i.e. there exist maps $\Gamma$ and $\Delta$ such that
\begin{equation}\label{r_p_coder_controller}
\begin{split}
s_0^{\tau-1} &\equiv \Gamma (x_0) \\
u_0^{\tau-1} &\equiv \Delta ( s_0^{\tau-1} ).
\end{split}
\end{equation}

Now consider the disjoint regions $\Gamma^{-1}(c_{0}^{\tau -1}) \subseteq X$
as the symbol sequence $c_0^{\tau-1}$ varies over all possible sequences in $S_0 \times \dots \times S_{\tau-1}$.
The total number of distinct symbol sequences is $\prod_{k=0}^{\tau-1}\mu_k$, and
hence the total number $n$ of distinct coding regions does not exceed it.
Denote these coding regions by $C^1, \dots, C^n$ and note that $X= \cup_{i=1}^n C^i$.
We can then rewrite the control equation in (\ref{r_p_coder_controller}) as
\begin{equation}\label{new_r_p_coer_controller}
u_0^{\tau-1}=\Delta^*\big( C^i \big) \quad \mbox{if} \quad x_0 \in C^i
\end{equation}
by defining the map $\Delta^*(C^i)=\Delta(c_{0}^{\tau-1})$ iff $C^i=\Gamma^{-1}(c_{0}^{\tau-1})$.

We now construct the open cover $\alpha$ and the mapping sequence
$G={\{ G_k: \alpha \to U \}}_{k=0}^{\tau-1}$
required in the definition of RWTFE.
Before we define $\alpha$, observe that for any $x$ in coding region $C^i$,
\[ \Phi_{\Delta^*(C^i)}(x)\in \mbox{int}K,\]
where the left hand side denotes the dynamical map $F_{u_0^{\tau-1}}$ applied with the input sequence (\ref{new_r_p_coer_controller}).
By robust weak invariance, it then follows that for any $x \in C^i$ there is an open set $O(x)$ that contains $x$ and is such that
\[ \Phi_{\Delta^*(C^i)}(x_0)\in \mbox{int}K \quad \mbox{for any $x_0 \in O(x)$}. \]
We can construct for each $1 \leq i \leq n$ an open set $L^i=\cup_{x \in C^i}O(x)$ so that
\begin{equation*}\label{r_open_sets}
\Phi_{\Delta^*(C^i)}(x_0) \in \mbox{int}K \quad \mbox{for any $x_0 \in L^i$,}
\end{equation*}
which is equivalent to
\begin{equation*}
\Phi_{\Delta^*(C^i)}(L^i) \subseteq \mbox{int}K.
\end{equation*}
As $C^i \subseteq L^i$ and $X = \cup_{i=1}^n C^i$, we have an open cover $\alpha=\{ L^1, \dots, L^n \}$ for $X$.
Finally, construct the mapping sequence $G$ on $\alpha$ by
\[ G_k(L^i)=\,\mbox{the $k$th element of $\Delta^*(C^i)$ \quad for each $L^i \in \alpha$ and $0 \leq k \leq \tau -1$}.    \]
\noindent It is evident that constraint (RC) on $(\alpha, \tau, G)$ is satisfied.

By construction we know that $N(\alpha)\leq n$ and we obtain
\begin{equation*}
\begin{split}
h_{\mrrw}&:=\inf_{(\alpha, \tau, G)\in\mathcal{RC}}\frac{\log_2 N(\alpha)}{\tau} \\
&\leq \frac{\log_2 N(\alpha)}{\tau}\\
&\leq\frac{\log_2 n}{\tau} \\
&\leq \frac{\log_2 \left (\prod_{k=0}^{\tau-1} \mu_k\right )}{\tau}=\frac{1}{\tau}\sum_{k=0}^{\tau-1}\log_2 \mu_k \leq R+\eps.
\end{split}
\end{equation*}
Since $\eps$ was arbitrary, we have the desired result.

\paragraph{Achievability of the Lower Bound.}

To prove that data rate arbitrarily close to $h_{\mrrw}$ can be achieved, we first show that any triple $(\alpha, \tau, G)$ that satisfies constraint (RC) with $K$ induces a coder-controller that renders $K$ robustly weakly invariant.
Given such a triple, define
\begin{equation*}\label{r_uninfimised_entropy}
H:=\frac{\log_2 N(\alpha)}{\tau}.
\end{equation*}
Recalling that $\alpha$ is an open cover of  compact set $X$, select and denote by
$\{ D^1, \dots, D^m \}$ a minimal subcover of $\alpha$, where $m=N(\alpha)$.

We construct a $\tau$-periodic coding law using these overlapping sets via the rule
\begin{equation*}\label{r_coding_rule}
s_k=
\begin{cases}
\min \{ i: x_0 \in D^i \}& \mbox{if $k =0$} \\
0 &\mbox{otherwise}
\end{cases}, \  \ 0\leq k\leq \tau -1.
\end{equation*}
The coding alphabet is of size $\mu_k=m$ when $k=0$ and $1$ otherwise.

The next step is to construct the controller from the mapping sequence $G$.
Upon receiving the symbol $s_{0}=i$ that indexes an open set $D^i$ in the minimal subcover of $\alpha$,
containing $x_{0}$, the controller applies the input sequence
\begin{equation*}\label{r_achiev_periodic_controller}
u_0^{\tau-1} =G(D^i)
\end{equation*}
By assumption $(\alpha, \tau, G)$ satisfies constraint (RC).
From this it is easy to see that $K$ is robustly weakly invariant with the coder-controller and $\tau$ defined above.

By (\ref{r_coding_rule}), the average data rate over the cycle is
\begin{equation*}
\bar{R} =\frac{\log_2 m}{\tau}=\frac{\log_2 N(\alpha)}{\tau} \equiv H.
\end{equation*}
As $h_{\mrrw}$ is the infimum of $H$, for any $\eps >0$ we can find $(\alpha, \tau, G)$ yielding $H < h_{\mrrw}+\eps$.
Hence $\bar{R}<h_{\mrrw}+\eps$. The result follows by observing that $R\leq\bar{R}$ by
definition (\ref{r_data_rate})  and then choosing $\eps$ arbitrarily small.

\section{Conclusion}
In this paper we used open set techniques to propose two extensions of topological feedback entropy (TFE)
; one is to continuous plants with continuous, partial observations, and the other to a class of discontinuous plants with full state observations.
In each case we showed that the extended TFE coincides with the smallest average bit-rate
between the plant and controller that allows weak invariance to be achieved,
thus giving these concepts operational meaning in communication-limited control.

Our focus here has been on formulating the concepts of TFE rigorously
and showing the fundamental limitations they impose on the communication rates needed for control.
Future work will focus on computing upper and lower bounds on them for various system classes,
and on strengthening the weak-invariance objective to controlled invariance.

Another important question left for future work is: how to construct a unified notion of TFE for
discontinuous, partially observed plants.
Such a notion would be an important step toward a theory  of information flows
for cooperative nonlinear control systems interconnected by digital communication channels.

\appendix


\section{Proof for Proposition \ref{output_feasibility}}
Assume (Ob) and (WCI).
Suppose that $s$ and $v_0^{s-1}$ satisfy (Ob).
We construct $\alpha, \tau, G$ that satisfy (C) as follows.
Set $\tau$ equal to $t$ in (WCI).
For $\alpha$, we first observe the following.
By (WCI), the continuity of $f_u$, and openness of $\mbox{int}K$, for any $x_s \in X$ there is an open set $O(x_s)$ in $X$ that contains $x_s$ and is such that $f_{H_{\tau-1}(x_s)} \dots f_{H_0(x_s)}(O(x_s)) \subset \mbox{int}K$.
By ranging $x_s$ in $f_{v_0^{s-1}}(X)$ we obtain an open cover $\{ O(x_s): x_s \in f_{v_0^{s-1}}(X) \}$ of $f_{v_0^{s-1}}(X)$ in $X$.

We then use the continuous map $h_{v_0^{s-1}}:=f_{v_0^{s-1}}g_{v_0^{s-1}}^{-1}$ to form the collection of open sets in $g_{v_0^{s-1}}(X)$, $\{ h_{v_0^{s-1}}^{-1}(O(x_s)): x_s \in f_{v_0^{s-1}}(X) \}$.
Note that the union of sets $h_{v_0^{s-1}}^{-1}(O(x_s))$ is $g_{v_0^{s-1}}(X)$.
Since $g_{v_0^{s-1}}(X)$ is equipped with the subspace topology of $Y^{s+1}$, each open set $h_{v_0^{s-1}}^{-1}(O(x_s))$ in $g_{v_0^{s-1}}(X)$ can be expanded to an open set $S'(x_s)$ in $Y^{s+1}$ so that $h_{v_0^{s-1}}$ maps points in $S'(x_s) \cap g_{v_0^{s-1}}(X)$ into $O(x_s)$.
Now $\{S'(x_s): x_s \in f_{v_0^{s-1}}(X) \}$ is an open cover of $g_{v_0^{s-1}}(X)$ in $Y^{s+1}$.

The open cover $\alpha$ of $g_{v_0^{s-1}}(X)$ in $Y^{s+1}$ is then chosen to be a finite subcover $\{ L^1, \dots, L^r \}$ of $\{S'(x_s)\}$.
The existence of a finite subcover is guaranteed by the compactness of $g_{v_0^{s-1}}(X)$ as the image of the compact set $X$ under the continuous map $g_{v_0^{s-1}}$.
Let $x_s^q$ be the point in $f_{v_0^{s-1}}(X)$ such that $L^q=S'(x_s^q)$, and let
\begin{equation*}
G_k(L^q)=H_k(x_s^q), \quad 1 \leq q \leq r, \, 0 \leq k \leq \tau-1.
\end{equation*}
\noindent By construction we have that if $x_0 \in g_{v_0^{s-1}}^{-1}(L^q)$ then
\begin{equation*}
f_{v_0^{s-1}G(L^q)} (x_0)
= f_{v_0^{s-1}{\{ H_k(x_s^q)\}}_{k=0}^{\tau-1}}(x_0) \in \mbox{int}K.
\end{equation*}
This confirms the feasibility of (C) under (Ob) and (WCI).

\section{Proof for Proposition \ref{B_j_properties}}
The openness of $B_j$ is confirmed by writing it as
 \begin{multline*}
B_j=g_{v_0^{s-1}}^{-1}(A_0) \cap \Phi_{G(A_0)}^{-1} g_{v_0^{s-1}}^{-1}(A_1) \cap \Phi_{G(A_0)}^{-1} \Phi_{G(A_1)}^{-1} g_{v_0^{s-1}}^{-1}(A_2) \cap \cdots \\
 \dots \cap \Phi_{G(A_0)}^{-1} \dots \Phi_{G(A_{j-2})}^{-1} g_{v_0^{s-1}}^{-1}(A_{j-1}),
 \end{multline*}
where $\Phi_{G(A)}:= f_{v_0^{s-1}G(A)}$ is a continuous map.

Since $g_{v_0^{s-1}}$ is continuous and $\alpha$ is an open cover of $g_{v_0^{s-1}}(X)$ in $Y^{s+1}$, the collection $\{ g_{v_0^{s-1}}^{-1}(A) : A \in \alpha \}$ is an open cover of $X$.
Hence any $x_0 \in X$ must be in some $g_{v_0^{s-1}}^{-1}(A_0)$, $A_0 \in \alpha$.
Then constraint (C) implies that the $s+\tau$ inputs $v_0^{s-1}G(A_0)$ forces $x_{s+\tau} \in \mbox{int}K \subset X$.
Repeating this process indefinitely, we see that for any $x_0 \in X$ there is a sequence $A_0, A_1, \dots$ of sets in $\alpha$ such that $x_{i(s+\tau)} \in g_{v_0^{s-1}}^{-1}(A_i)$ when the input sequences $u_{i(s+\tau)}^{(i+1)(s+\tau)-1}$
are used.

\section{Proof for Lemma \ref{subadditivity}}
For each $j, k \in \NN$, the collection $\beta_{j+k}$ consists of all sets of the form
\begin{multline*}
B_{j+k}=\left[ g_{v_0^{s-1}}^{-1}(A_0) \cap \Phi_{G(A_0)}^{-1} g_{v_0^{s-1}}^{-1}(A_1) \cap \dots \cap \Phi_{G(A_0)}^{-1} \dots \Phi_{G(A_{j-2})}^{-1} g_{v_0^{s-1}}^{-1}(A_{j-1}) \right] \\
        \cap \Phi_{G(A_0)}^{-1} \dots \Phi_{G(A_{j-1})}^{-1} \Big( g_{v_0^{s-1}}^{-1}(A_j) \cap \Phi_{G(A_j)}^{-1} g_{v_0^{s-1}}^{-1}(A_{j+1}) \cap \cdots \\
        \dots \cap \Phi_{G(A_j)}^{-1} \dots \Phi_{G(A_{j+k-2})}^{-1} g_{v_0^{s-1}}^{-1}(A_{j+k-1}) \Big),
\end{multline*}
where $\Phi_{G(A)}= f_{v_0^{s-1}G(A)}$ with $A_0, \dots, A_{j+k-1}$ ranging over $\alpha$.
Note that the expression inside the square brackets runs over all sets in $\beta_j$, while the expression inside the large parentheses runs over all sets in $\beta_k$.
Constrain ${\{ A_i \}}_{i=0}^{j-1}$ to index sets in a minimal subcover $\beta_j'$ of $\beta_j$, and ${\{ A_i \}}_{i=j}^{j+k-1}$ to those in a minimal subcover $\beta_k'$ of $\beta_k$.
Denote the constrained family of sets $B_{j+k}$ thus formed by $\beta_{j+k}^{*}$.

We claim that $\beta_{j+k}^{*}$ is still an open cover for $X$.
To see this, observe that any $x_0 \in X$ must lie in a set $B_j' \in \beta_{j}'$ indexed by some sequence ${\{ A_i \}}_{i=0}^{j-1}$ in $\alpha$.
Furthermore, $\Phi_{G(A_{j-1})} \dots \Phi_{G(A_0)}(x_0) \in \mbox{int}K \subset X$ and thus lies in some set $B_k'$ in the minimal subcover $\beta_k'$.
Hence $x_0 \in \Phi_{G(A_0)}^{-1} \dots \Phi_{G(A_{j-1})}^{-1} (B_k')$ and so $\beta_{j+k}^{*}$ is still a cover for $X$.
As there are $N(\beta_j)$ sets in $\beta_j'$, and to each there correspond $N(\beta_k)$ possible sets in $\beta_k'$, the number of distinct elements in $\beta_{j+k}^{*}$ is less than or equal to $N(\beta_j) N(\beta_k)$.
By the definition of minimal subcovers we have that for any $j, k \in \NN$,
\begin{equation*}
N(\beta_{j+k}) \leq N(\beta_j) N(\beta_k).
\end{equation*}
Take logarithm with base $2$ on each side of the inequality.

\bibliographystyle{spmpsci}      

\bibliography{rep1ref}

%
%

\end{document}